\documentclass[12pt,oneside]{amsart}
\usepackage[utf8]{inputenc}
\usepackage[english]{babel}
\usepackage{amsfonts}
\usepackage{amssymb, amsmath}
\usepackage{amsthm}
\usepackage{graphicx}
\usepackage{comment}
\usepackage{pgf}
\usepackage{graphics,epsfig, subfigure}
\usepackage{url}
\usepackage{srcltx}
\usepackage{hyperref}
\usepackage{enumerate}
\usepackage{caption}
\usepackage{xcolor}

\theoremstyle{plain}
\newtheorem{theorem}{Theorem}
\newtheorem*{theorem*}{Theorem}
\newtheorem{theoremK}{Theorem}

\newtheorem{proposition}{Proposition}
\newtheorem*{proposition*}{Proposition}
\newtheorem{lemma}{Lemma}

\newtheorem*{problem*}{Problem}

\newtheorem*{remark*}{Remark}
\theoremstyle{definition}

\newcommand{\R}{\mathbb R}

\newcommand{\E}{\mathbb E}

\newcounter{reminder}

\title
[Characterizations of polyhedrons]
{On some characterizations of convex polyhedra}
\author{Sergii Myroshnychenko}
\address{Department of Mathematical and Statistical Sciences, University of Alberta,
	Edmonton, Alberta T6G 2G1 , Canada.} \email{myroshny@ualberta.ca}

\begin{document}
	
	\thanks{The author is supported by PIMS Postdoctoral Fellowship.}
	\begin{abstract}
This work provides two sufficient conditions in terms of sections or projections for a convex body to be a polytope. These conditions are necessary as well.
	\end{abstract}
	
	\maketitle
\section{Introduction}

Many curious properties of convex bodies can be determined by two dual notions: projections and sections. This paper contains characterizations of polytopes in terms of non-central sections as well as point projections (see Figure \ref{fig:mainfig}). 

\begin{theorem}\label{th2}
Let $K$ be a convex body in $\E^d, d \geq 3$, and $\{H_{\alpha}\}_{\alpha \in \mathcal{A}}$ be a set of $k$-dim affine spaces, $2 \leq k \leq d-1$, all of which intersect the interior of $K$, such that:
\begin{itemize}
	\item for any supporting line $l$ of $K$, there exists a plane $H_{\alpha} \supset l$;
	\item for all $\alpha \in \mathcal{A}$, the intersection $K \cap H_{\alpha}$ is a $k$-dim polytope.
\end{itemize}
Then $K$ is a polytope.
\end{theorem}

\begin{figure}[h!]
	\centering
	\includegraphics[scale=0.25]{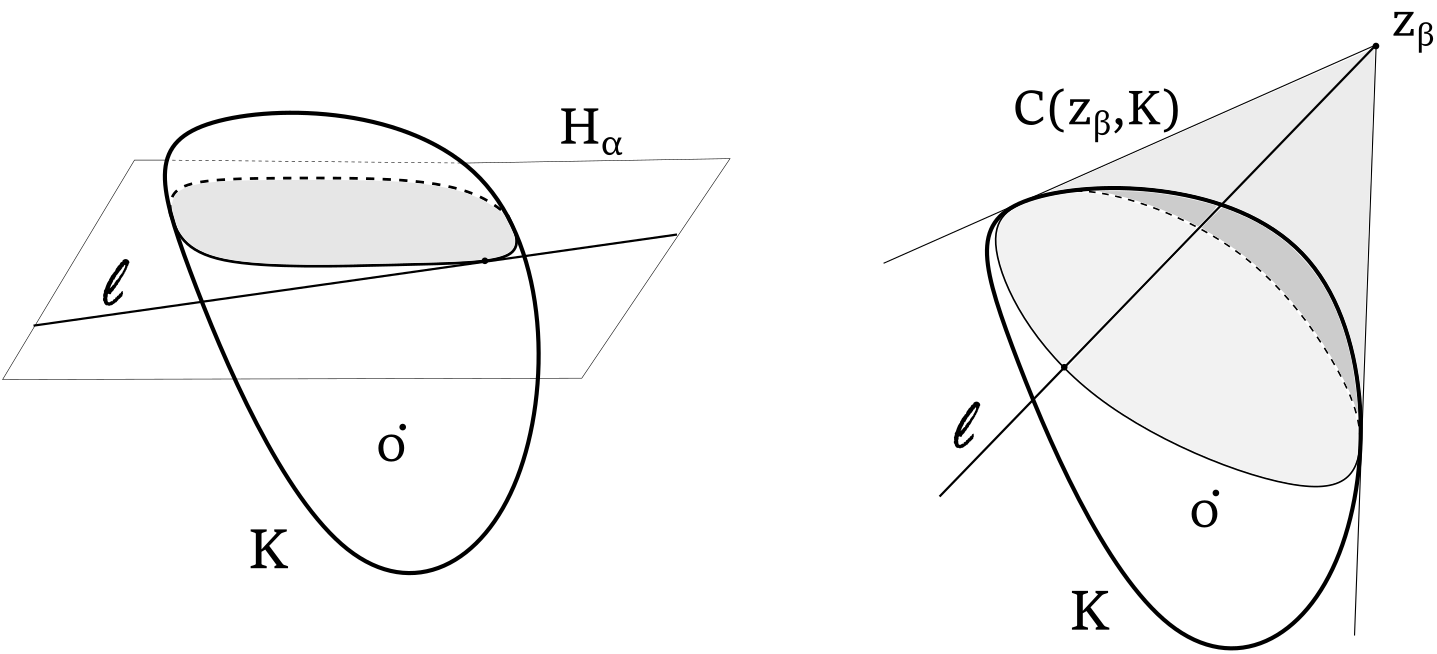}
	\captionsetup{justification=centering}
	\caption{A $(d-1)$-dim section of $K$ by hyperplane $H_{\alpha}$, and a $d$-dim visual cone $C(z_{\beta},K)$ of $K$.}
	\label{fig:mainfig}
\end{figure}

For instance, if $k=d-1$ and $\delta$ is a continuous function on $S^{n-1}$, then set of hyperplanes $\big\{H_{\xi}\big \}_{\xi}$ intersecting the interior of $K$,
\[
H_{\xi} = \big \{x \in \E^d: \, x \cdot \xi = \delta(\xi) \big \}, \quad \xi \in S^{d-1},
\]
satisfies the conditions of the theorem for sufficiently small $\delta$ (see \cite{BG}). In particular, for $\delta \equiv 0$, Theorem~\ref{th2} implies the celebrated result of Victor Klee from 1959 
\begin{theoremK}[\cite{K}] \label{K1}
	A bounded convex subset of $\E^d$ is a polytope if any of its $k$-dim central sections, $2 \leq k \leq d-1$, is a polytope.
\end{theoremK}

The version of Theorem \ref{th2} for ellipsoids was handled in \cite{BG}. It also implies the corresponding result for Euclidean balls. We note that such settings are considered in several problems of Convex Geometry, such as questions related to characterizations of balls by sections and caps (\cite{KO}), conical sections (\cite{RY}), floating bodies (\cite{B}, \cite{BSW}),  $t$-sections (\cite{Y}, \cite{YZ}). Additionally, we also provide the following dual result for visual cones  (see definition \ref{sightCone})
\begin{theorem}\label{th}
	Let $K$ be a convex body in $\E^d, d \geq 3$, and $\{z_{\beta}\}_{\beta \in \mathcal{B}} \subset \E^d$ be a set of exterior points of $K$ that satisfies:
	\begin{itemize}
		\item for any supporting line $l$ of $K$, there exists a point $z_\beta \in l$;
		\item for a fixed $k$, $3 \leq k \leq d$, and all $\beta \in \mathcal{B}$, any $k$-dim visual cone $C_k(z_{\beta}, K)$ is polyhedral.
	\end{itemize}
	Then $K$ is a polytope.
\end{theorem}

For example, a closed surface $S$ containing $K$ in its interior satisfies the conditions of the point-set in Theorem \ref{th}. Also, when $S$ is a convex surface, the analogous problem for circular cones was solved in \cite{M}.
In the class of elliptical cones, where $\{z_{\beta}\}_{\beta \in \mathcal{B}} $ is a closed set, the corresponding result was obtained for ellipsoids in \cite{BG}. 
Two pertinent characterizations of ellipsoids regarding visual cones were considered in \cite{GO}. 
Questions related to measures of visual cones rather than shapes were also studied in \cite{Ku} and \cite{KO}.
A resembling construction for illumination bodies is discussed in \cite{MW};
for the related well-known illumination problem see \cite{Bo} and \cite{BH}. A question regarding the visual recognition of polytopes was also investigated in \cite{My}.

Lastly, by polar duality (\cite{Ga}, p. 22), Theorem \ref{th2} allows one to extend Theorem \ref{th} to the case of infinitely distant points. This way, we obtain Klee's Theorem for orthogonal projections
\begin{theoremK}[\cite{K}]\label{K2}
	A bounded convex subset of $\E^d$ is a polytope if any of its $k$-dim orthogonal projections,  $2 \leq k \leq d-1$, is a polytope.
\end{theoremK}

The main idea of the proof for Theorems \ref{th2} and \ref{th} is to show that, under the provided conditions, extreme points of $K$ cannot accumulate. To prove Theorem \ref{th2}, we generalize the construction from \cite{Z}. The proof of Theorem \ref{th} relies on a similar idea for orthogonal projections on two-dimensional subspaces.

\section{Preliminary definitions and results }
In this section we provide the results, definitions, and notation for notions used throughout the paper. For more details on these and related concepts of Convex Geometry, see \cite{Ga}, \cite{Gr}, \cite{S}.

By lowercase letters such as $p, q, r$ etc., we denote points in $d$-dim Euclidean space $\E^d$, $d \geq 3$. 
For any two points $p, q$, the closed segment connecting them is denoted by $[p \, q]$, and its interior is $( p\, q)$. Notation $\angle (p \, q\, r)$ stands for the angle between $[p \, q]$ and $[r\, q]$. The Euclidean length of $[p \, q]$ is $\|p -q\|$.
Then, the unit sphere with center at the origin $o$ is $S^{d-1} = \{x \in \E^d: \, \|x \| = 1\}$.

For a set $V \subset \E^k \subseteq \E^d$, $k \leq d$, by $int \, V$ we mean the set of relative interior points of $V$, and $\partial V$ stands for its set of boundary points in $\E^k$. A set $V \in \E^d$ is called \textit{convex} if for any $p, q \in V$, one also has $[p \, q] \subset V$. By a \textit{convex body} we mean a compact convex set with non-empty interior.
A line $l$ is called a \textit{supporting line} of a set $V$ if $l \cap V \subset \partial V$. A \textit{supporting ray} is a half-line of a supporting line that has non-empty intersection with the set.
For a convex set $V$, a point $r\in V$ is called \textit{extreme} if there do not exist two distinct points $p, q \in V$, such that $r \in (p \, q)$. The distance between point $p$ and a set $V$ is
$$
dist \, (p, V) = \inf  \limits_{z \in V} \|z-p\|.
$$

The\textit{ convex hull }of a finite number of sets $V_1, \ldots, V_m  \subset \E^d$ is
\[
conv \, \Big\{V_1, \ldots, V_m \Big \} = \Big\{\lambda_1 p_1 + \ldots + \lambda_m p_m: \, \sum \limits_{j=1}^m \lambda_j = 1,\, \lambda_j \geq 0, \, p_j \in V_j\Big\}.
\]
The convex hull of a finite number of points is called \textit{a (convex) polytope}. In this regard, Minkowski was first to show
\begin{proposition*}[\cite{Gr}, p.75]
	Every convex set in $\E^d$ is the convex hull of its extreme points.
\end{proposition*}
In particular, it implies that a polytope is the convex hull of the finite set of its extreme points that are called \textit{vertices}.

By \textit{polyhedral cone }$C$ with \textit{apex} at the origin, we understand a non-empty intersection of a finite family of closed half-spaces. 
By a translation, we extend this notion to a cone with the apex at any point $z \in \E^d$.

The natural notion of a full-dimensional \textit{visual} (also called \textit{sight} in \cite{M})  \textit{cone} of a set (see Figure \ref{fig:mainfig} and Theorem \ref{th}), $V \subset \E^d$ with apex at $z \not \in V$ is
\begin{equation}\label{sightCone}
C(z,V) := \Big \{z + t \cdot (x-z):\, x \in V, \, t \geq 0\Big \}.
\end{equation}

For a set $V$ with non-empty relative interior, consider a $k$-dim affine subspace $H$, $2 \leq k \leq d-1$, such that $z \in H$ and $H \cap int \, V \neq  \emptyset$. Then, a $k$\textit{-dim visual cone} with the apex at $z$ is a sub-cone of $C(z,V)$,
$$
C_k(z, V) = C(z,V) \cap H.
$$
In this regard, observe the following result of Mirkil
\begin{lemma}[\cite{Mi}]\label{polycones}
	A cone $C \subset \E^d, d > 3,$ with the apex at the origin is polyhedral if and only if, for any $k$-dim subspace $H$, $3 \leq k \leq d-1$, 
	the cone $C_H = C \cap H$ is either empty or
	polyhedral.
\end{lemma}

\section{Proof of Theorem \ref{th2}}

To prove Theorem \ref{th2}, we need a couple of auxiliary results. The case $d=3$ of the following considerations for central sections was shown in \cite{Z}. Here we provide their generalizations for non-central cases in $d \geq 3$.

\begin{lemma}\label{diamond}
	Let $K \subset \E^d$ be a convex body and $Q$ be a $k$-dim convex set, $Q \subset \partial K$, $1 \leq k \leq d-2$. Let $p,q \in K$ be two distinct points, such that $[p \,q] \cap Q = \{x\} \in (p \, q)$, then
	$$
	Q^{pq} = conv \, \Big\{Q, [p \, q]\Big\} \subset \partial K.
	$$
\end{lemma}

\begin{figure}[h!]
	\centering
	\includegraphics[scale=0.2]{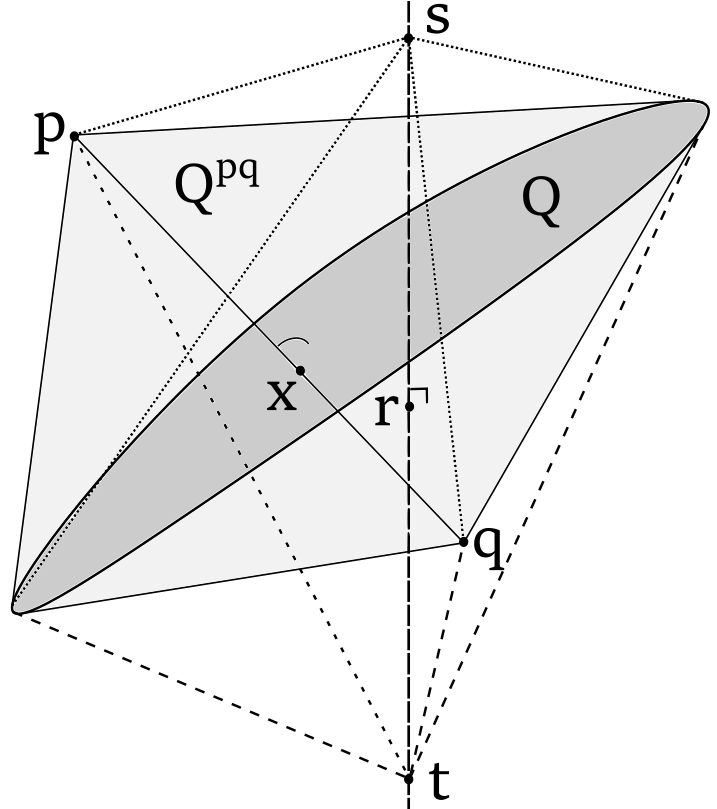}
	\label{fig:diamond}
\end{figure}

\begin{proof}
	
	Assume the opposite, that there exists a point $r \in Q^{pq}$, such that $r \in int \, K$. Then, there exists a line segment $[t\, s] \subset int \, K$, such that $r \in (t \, s)$ and $[t \, s]  \perp Q^{pq}$. This follows from the observation that $r \in int \, K$ is contained in a ball of small enough radius contained in $int \, K$ completely, so $t,s \not \in Q$ can be chosen from its boundary sphere. Hence,
	$$
int \, Q \subset int \, Q^{pq} \subset int \, \bigg(conv\, \Big\{ Q^{pq},t,s\Big\}\bigg) \subset int\, K.
	$$
	
	However, this contradicts the original assumption $int \, Q \subset Q \subset \partial K$.
\end{proof}

The next proposition shows that the conditions of the theorem prevent extreme points of $K$ from ``concentrating along'' a segment in $K$.

\begin{proposition}\label{cone}
	Suppose that a convex body $K$ satisfies the conditions of Theorem \ref{th2}. Let $p,q \in K$ be two distinct points, then there exists $\varepsilon = \varepsilon(p,q) > 0$ such that any $y \in K$ is not an extreme point of $K$ if
$$
	 0<\|p - y\| < \varepsilon, \quad \text{and} \quad \angle ypq < \varepsilon.
$$

\end{proposition}

\begin{proof}
	We consider two possible cases.
	\begin{itemize}
		\item[\underline{Case 1.}] $[pq] \cap int \, K \neq \emptyset$.
		
Assume that there exists a point $x \in (p \, q) \cap int\, K$ (see Figure \ref{fig:pr1}). Then, for a small enough $R = R(x)>0$, ball $B(x,R)$ of radius $R$ centered at $x$ belongs to the interior of $K$,  $B(x,R) \subset int \, K$, and $p,q \not \in B(x,R)$. Hence, we have 
		$$
		(p \, q) \subset conv\, \Big\{B(x,R), p, q\Big\} \subset int \, K.
		$$ 
		
		Thus, we can choose $x$ to be the mid-point of $[p \, q]$ and an $\varepsilon >0$ such that
		$$
		\varepsilon  <	\min \bigg\{\sqrt{\|p-x\|^2 - R^2},  \arcsin \frac{R}{\|p-x\|}\bigg\}.
		$$
		
\begin{figure}[h!]
	\centering
	\includegraphics[scale=0.3]{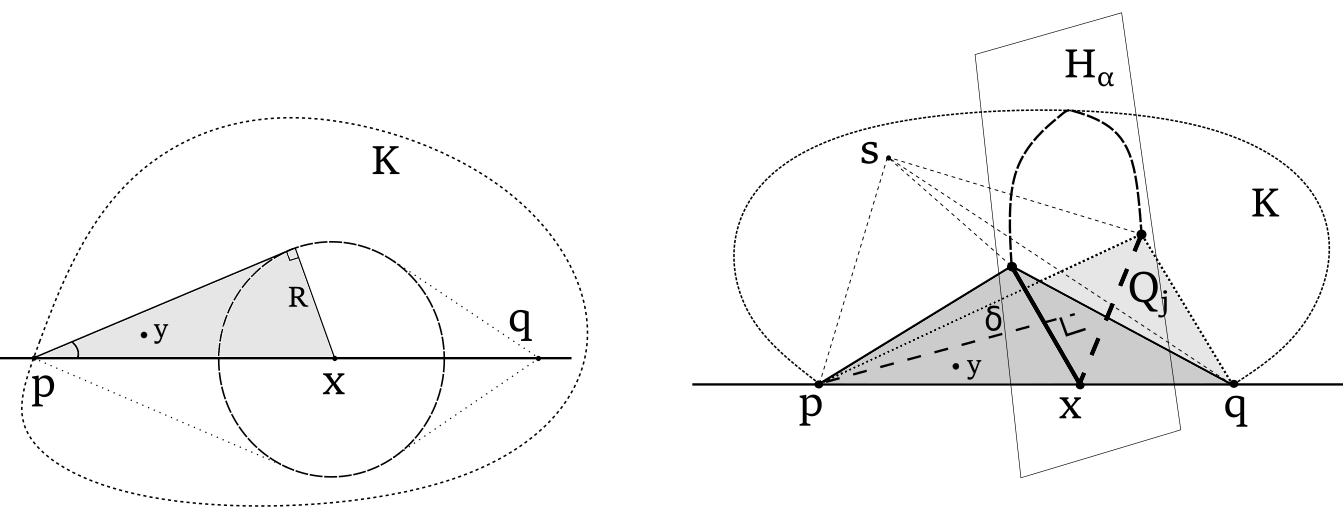}
	\captionsetup{justification=centering}
	\caption{$K$ has no extreme points too close to point $p$ and segment $[p \, q]$. }
	\label{fig:pr1}
\end{figure}

		\item[\underline{Case 2.}] $[p \, q] \subset \partial K.$ 
		
		Let $H_{\alpha}$ be a plane that intersects $[p \, q]$ transversally at the mid-point $x$. Then $x \in \partial P_{\alpha}$ for the polytope $P_{\alpha} = K \cap H_{\alpha}$. Let $Q_1, \ldots, Q_m$ be all the facets of $P_{\alpha}$ that contain $x$, and $\Big\{x_1, \ldots, x_N\Big\}$ be the set of all of their vertices. If $x$ is one of these vertices, exclude it from the set.
By Lemma \ref{diamond}, 
$$
Q_j^{pq}=conv \, \Big\{ Q_j, [p \, q] \Big \} \subset \partial K, \quad \forall j=1,\ldots, m.
$$

Let  a point $s \in int \, K$, $\delta = dist \, (p, H_{\alpha}) > 0$, and 
$$
0 < \varepsilon < \min \Big\{\delta, \min_{j=1,\ldots,N}\big\{ \angle x_jpq\big\}\Big\}.
$$
Then, for any point $y \in K$, such that $\|p - y\| < \varepsilon$ and $\angle ypq< \varepsilon$, we have
$$
y \in conv\, \bigg \{x_1, \ldots, x_N, p, q, s \bigg\}.
$$
Thus, $y$ cannot be an extreme point of $K$.
\end{itemize}
\end{proof}

\begin{remark*}
	 The previous proposition shows that a small enough cone of revolution with apex $p$ and axis of rotation parallel to $[p \, q]$ cannot have any extreme points of $K$ in a small enough neighbourhood of $p$.
\end{remark*}

Now we have all the necessary ingredients to prove Theorem \ref{th2}.

\begin{proof}
	Suppose to the contrary that $\{q_n\}$ is an \textit{infinite} set of distinct extreme points of $K$. In particular, this implies that $\{q_n\}$ is bounded. By  Bolzano–Weierstrass theorem, we can consider a convergent subsequence
	$$
	\{q_{n_k}\} \to p, \quad \textrm{for} \quad n_k \to \infty.
	$$
	Then, $	p \in \partial K$, otherwise, any point in a small neighbourhood of $p$ cannot be extreme.	
	Again, consider a convergent subsequence $\{q_{{n_k}_s}\}$  such that the unit vectors
	$$
	u_{{n_k}_s} = \frac{q_{{n_k}_s} - p}{\|q_{{n_k}_s} - p\|} \to u, \quad \textrm{for} \quad n_{k_s} \to \infty.
	$$

	Let $l$ be the ray passing through $p$ in the direction of $u$. Our last step is to show that, under the assumptions of Theorem \ref{th2}, $l$ must intersect $K$ at more than a single point $p$.  Note that $l$ is a supporting ray of $K$ (otherwise, conditions on $p$ and $u$ are not satisfied). Pick a plane $H$ from the given family that contains $l$.
	Let $m$ be a supporting ray of $K \cap H$ that passes through $p$, contains a non-trivial segment $[p\,q] \subset \partial K$, and has the least angular distance to $l$ (by the angular distance we mean the angle between directional vectors of the rays). The existence of $[p \, q]$ follows from the fact that $K \cap H$ is a polytope.

	  Let $G$ be a hyperplane containing $l$, and such that $H$ is parallel to the normal vector of $G$. Also, let $E_n$ be a hyperplane orthogonal to ray $l$ and passing through point $q_n$ (see Figure \ref{fig:lipsc}).
	 Then let $s_n$ be the orthogonal projection of $q_n$ onto $G$; $b_n$ be the orthogonal projection of $s_n$ (or $q_n$) onto $l$; $b_n$ is the orthogonal projection of $a_n$ on $G$ (or on $l$), $t_n$ is the orthogonal projection of $q_n$ on $H$. Lastly, $a_n =E_n \cap [p\,q]$. Denote
	 \begin{align*}
	&\|p-q_n\| = \varepsilon_1, \quad \angle (q_n \, p \, b_n) = \varepsilon_2, \quad \angle (q_n\, b_n \, s_n) = \xi_n, \\
	& \quad  \angle (a_n \, p \, b_n) = \gamma, \qquad \angle (q_n \, a_n \, t_n) = \phi_n,
	 \end{align*}
 where $\varepsilon_1 \to 0^+, \varepsilon_2 \to 0^+$.
 
 	  \begin{figure}[h!]
 	\centering
 	\includegraphics[scale=0.4]{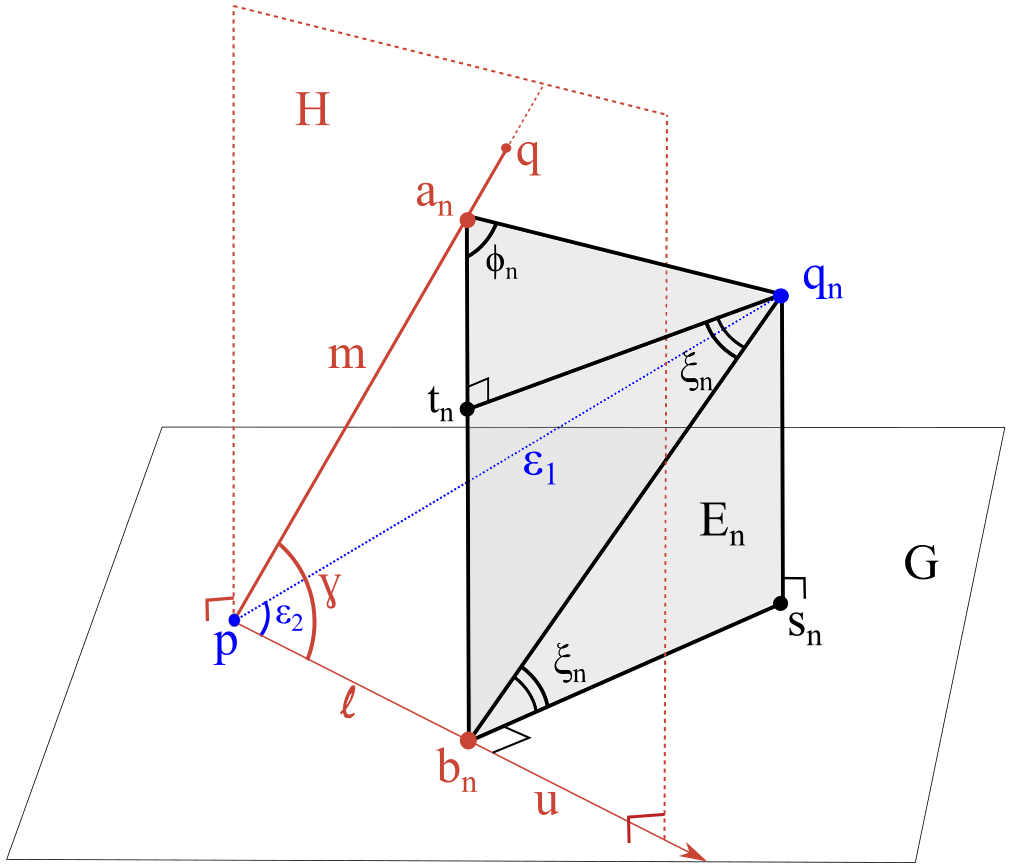}
 	\caption[]{Segment $[p\,q_n]$ is ``drifting" to align with $l$.}
 	\label{fig:lipsc}
 \end{figure}

	 Our goal is to show that $\gamma \equiv 0$. For $0 <\varepsilon < \min\{\varepsilon_1, \varepsilon_2\}$, we have 
	 \begin{align*}
	 &\|b_n - q_n\| = \varepsilon  \sin \varepsilon, \quad \|b_n - p\| = \varepsilon \cos \varepsilon, \quad \cot \phi_n = \frac{\|a_n - t_n\|}{\|t_n - q_n\|},\\
	& \|a_n - b_n\| = \|b_n-p\| \tan \gamma = \varepsilon \cos \varepsilon \tan \gamma, \\
	& \|a_n - t_n\| = \|t_n - q_n\| \cot \phi_n = \|b_n - q_n\| \cos \xi_n \cot \phi_n= \varepsilon \sin \varepsilon \cos \xi_n \cot \phi_n,\\
	&\|b_n - t_n\| = \|b_n - q_n\| \sin \xi_n = \varepsilon \sin \varepsilon \sin \xi_n.
	 \end{align*}
	 Now, $\|a_n - b_n\| \leq \|a_n - t_n\| + \|t_n - b_n\|$, thus
	 \begin{align*}
	 \varepsilon \cos \varepsilon & \tan\gamma \leq \varepsilon \sin \varepsilon \cos \xi_n \cot \phi_n + \varepsilon \sin \varepsilon \sin \xi_n,\\
	& \tan \gamma \leq \tan \varepsilon \cos \xi_n \cot \phi_n + \tan \varepsilon \sin \xi_n.
	 \end{align*}

	 Since $K$ is a convex body, we may define the boundary $\partial K$ locally around $p$ as at most two graphs of convex functions $z = f(x)$, where $G$ can be chosen as the $x$-plane. To see this we note that, by convexity, any non-empty intersection of a normal line to $G$ with $K$ is a closed segment $[z_1 \, z_2]$, where $z_1, z_2 \in \partial \, K$. This correspondence between a point in $G$ and the boundary may not define $\partial K$ as two graphs only if $[z_1 \, z_2] \subset \partial K, z_1 \neq z_2$. Let $w$ be a line in $G$ passing through $p$ and such that $w \cap int \, K \neq \emptyset$. For a small enough neighbourhood of points in $\partial K$, the line $w$ is not parallel to their supporting lines. We may apply a sheaf transformation to ''skew`` $K$ along $w$ towards the interior. Under this affine transformation, $K$ remains convex, and no such othogonal segments contained in $\partial K$ in a small neigbourhood $p$ are possible.
 Luckily, convex functions are locally Lipschitz (\cite{RV}), which implies that
	 $$
	 \cot \phi_n \leq L, \quad L > 0.
	 $$
	 Here $L$ is chosen as the maximum of the two Lipschitz constants in case locally $\partial K$ is represented as two graphs.	Hence, for a constant angle $\gamma \geq 0$, $\lim \limits_{\varepsilon \to 0} \tan \gamma = 0$. We conclude that $\gamma \equiv 0$ and $l=m$.  In these settings, we apply Proposition \ref{cone} for the segment $[p \, q]$ to observe that $q_n$ cannot be extreme in a small neighbourhood of $p$, however the sequence of extreme points $ q_n \to p$ for $n \to \infty$.
	This yields a contradiction to the assumption on an infinite number of extreme points of $K$.
\end{proof}

\section{Proof of Theorem \ref{th}}
\subsection{Case $d=3$}

First, we prove that any extreme point of a $2$-dim projection of $K$ is an intersection of an edge of a visual cone with the $2$-dim plane of the projection. Then, we show that every $2$-dim projection of $K$ is a polygon. Thus, by Theorem \ref{K2} in $d=3$, we conclude that $K$ is a polytope.

\begin{proof}
Choose an arbitrary $\xi \in S^2$. For any $x \in \partial (K|\xi^{\perp})$, consider the line
$$
l_{x} = \{x + t \cdot \xi, \quad t \in \R\}.
$$
Then take $z_{\beta} \in l_x$ and consider the support cone $C(z_{\beta}, K)$ (see Figure \ref{boo}). 

\begin{figure}[h!]
	\centering
	\includegraphics[scale=0.4]{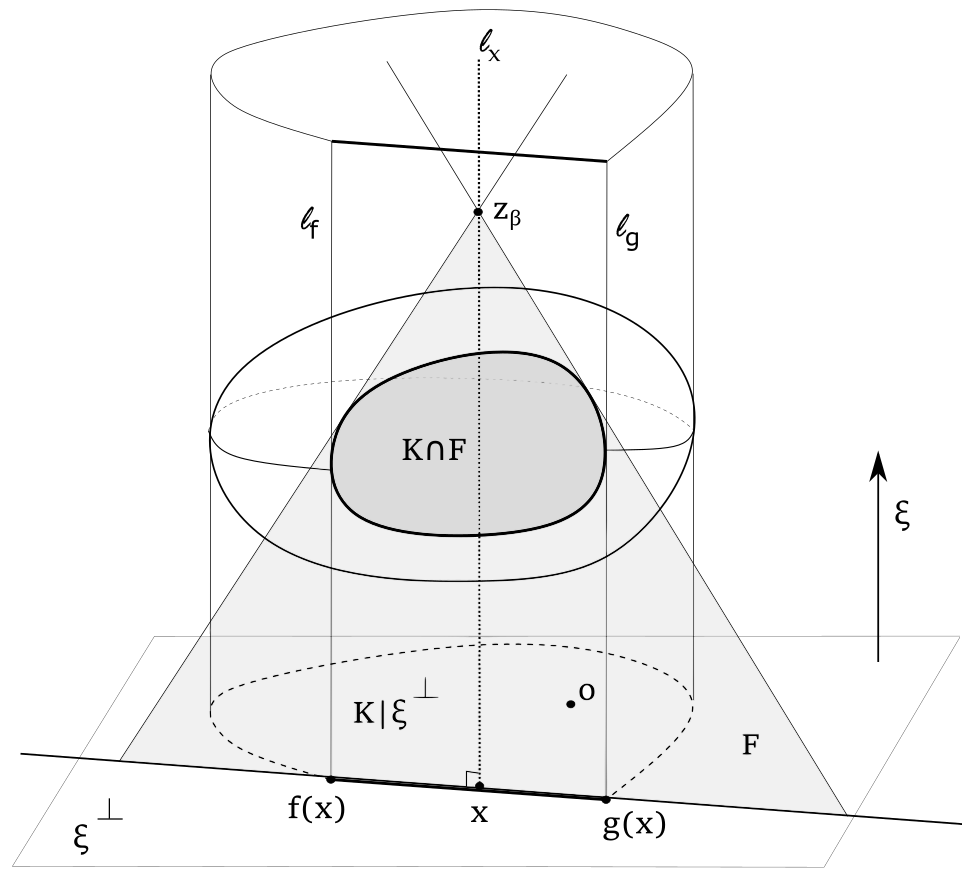}
	\caption{Distinguishing extreme points of $K| \xi^{\perp}$.}
	\label{boo}
\end{figure}

Assume that $l_x$ intersects the relative interior of the facet of cone $C(z_{\beta},K)$ contained in a plane $F$. Then $K \cap F$ is a convex subset of $F$, and consider two distinct lines $l_f, l_g$ parallel to $\xi$ that are supporting lines of $K \cap F$ in $F$. Denote 
$$
f(x) = l_f \cap \xi^{\perp}, \qquad g(x) = l_g \cap \xi^{\perp}.
$$
Hence, the orthogonal projection of $K \cap F$ onto $\xi^{\perp}$ is the non-degenerate segment $[f(x) \, g(x)]$. Which also implies that $x$ is not an extreme point of $K| \xi^{\perp}$. We repeat the same consideration for $g(x)$ and a point $z_{\beta'} \in S \cap l_g$ to note that $l_g$ may not intersect the relative interior of a facet of $C(z_{\beta'},K)$, otherwise $l_g$ is projected onto an interior point of $[f(x) \, g(x)]$. Thus, $l_g$ contains an edge of $C(z_{\beta'},K)$ and $g(x)$ is an isolated extreme point of $K| \xi^{\perp}$. We continue this procedure counterclockwise starting from $g(x)$ and preserving the notation to obtain a countable sequence of extreme points $\big\{v_j(x)\big\}_j$, where
$$
v_0(x) = x, \quad v_{j+1}(x) = g\big (v_{j}(x) \big), \quad j \geq 0.
$$

\begin{figure}[h]
	\centering
	\includegraphics[scale=0.15]{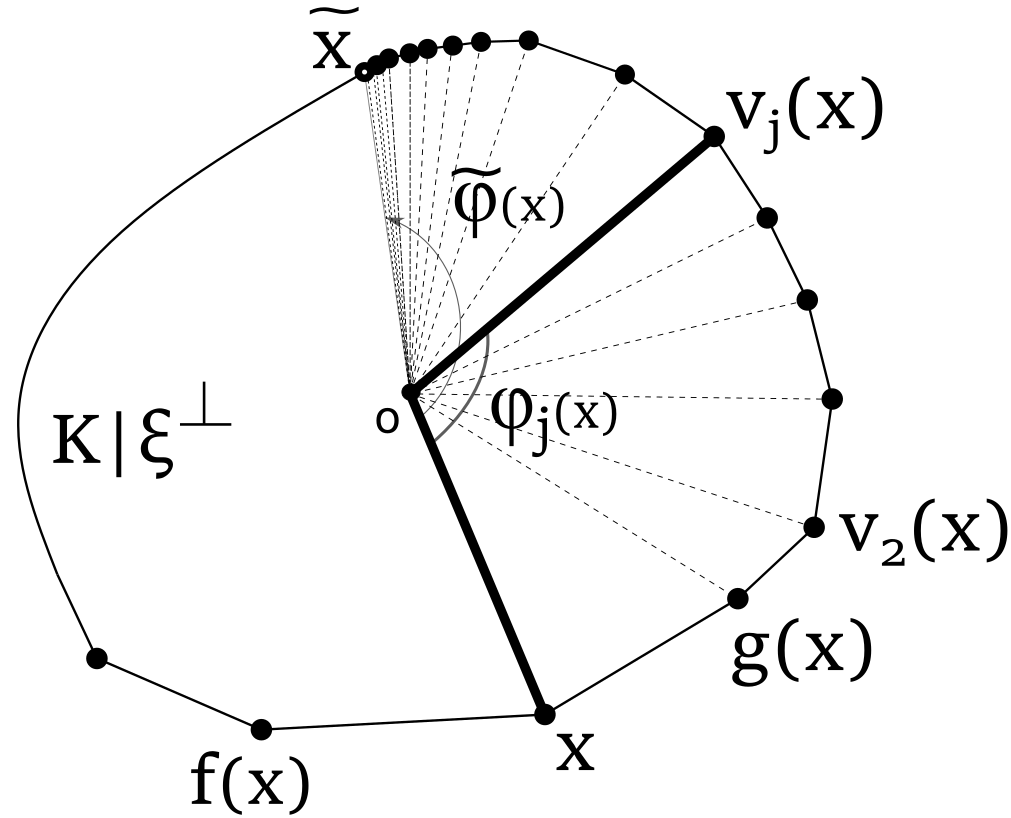}
	\caption{The extreme points of $K|\xi^{\perp}$ may not accumulate.}
	\label{koo-koo}
\end{figure}

We claim that $\big \{v_j(x) \big \}_j$ is finite. Assume the opposite, and
let us use a polar coordinate system in $\xi^{\perp}$ with the pole $o$ in the interior of the projection, and value $\varphi = 0$ corresponding to point $x$ (see Figure \ref{koo-koo}). This way we obtain a sequence of angles $\varphi_j(x) = \angle xo v_j(x)  $. Observe that $\{\varphi_j(x)\}$ is strictly increasing  (by the construction) and bounded ($0 < \varphi_j(x) < 2\pi$). The monotone convergence theorem implies that
$$
\exists \lim_{j \to \infty} \varphi_j(x) = \tilde{\varphi}(x),
$$
which corresponds to some point $\tilde{x} \in \partial(K| \xi^{\perp})$. Hence, any neighbourhood of $\tilde{x}$ in $\xi^{\perp}$ must contain an extreme point.

On the other hand, the closest extreme point to $\tilde{x}$ is at a fixed distance 
$$
\min\Big\{\|\tilde{x}-g(\tilde{x})\|, \|\tilde{x}-f(\tilde{x})\|\Big\}>0,
$$ 
regardless whether $\tilde{x}$ is an extreme point itself or not.
This yields a contradiction to the assumption that $\{v_j(x)\}$ is infinite. It implies that $K|\xi^{\perp}$ has a finite number of extreme points, i.e.,  $K|\xi^{\perp}$ is a polygon. Thus, by Theorem \ref{K2}, $K$ is a polytope.
\end{proof}

\subsection{Case $3 \leq k \leq d$}

\begin{proof} 
Consider any $3$-dim subspace $H$,  set  $S_H = \{z_{\beta}: \, z_{\beta} \in H\}$, and section $K_H = K \cap H$. By Lemma \ref{polycones}, for any point $z \in S_H$, cone $ C(z,K_H) = C_k(z,K) \cap H$ is polyhedral. Hence, the previous considerations apply, and projection of $K_H$ on any $2$-dim subspace of $H$ is a polygon.
Thus, by Theorem \ref{K2}, $K_H$ is a polytope. 
Since any $2$-dim subspace of $\mathbb{E}^d$ is a subspace of some $3$-dim
subspace $H$, by Theorem \ref{K2} we conclude that $K$ is a polytope as well.
\end{proof}

\section*{Acknowledgement}
The author would like to express his gratitude to Anton Petrunin for an enlightening conversation, to the anonymous Referee for the valuable remarks that improved the paper, and to the Pacific Institute for the Math Sciences (PIMS) for the continuous support.


\end{document}